\documentclass[10pt,reqno]{amsart}
\usepackage{amsmath, amssymb, amsfonts, amsthm, xcolor, mathdots}
\usepackage{mathtools}
\usepackage{enumerate}
\usepackage{hyperref}
\newtheorem{thm}{Theorem}[section]
\newtheorem{lem}[thm]{Lemma}
\newtheorem{prop}[thm]{Proposition}
\theoremstyle{definition}
\newtheorem{dfn}[thm]{Definition}

\newtheorem{conj}[thm]{Conjecture}

\newtheorem{remark}[thm]{Remark}
\theoremstyle{plain}
\newtheorem{cor}[thm]{Corollary}

\numberwithin{equation}{section}
\usepackage[a4paper, top = 1.2in, bottom = 1.2in, left = 1.2in, right = 1.2in]{geometry}
\numberwithin{equation}{section}

\newcommand{\C}{\mathbb{C}}
\newcommand{\N}{\mathbb{N}}

\newcommand{\Z}{\mathbb{Z}}
\newcommand{\F}{\mathbb{F}}

\newcommand{\mcA}{\mathcal{A}}

\newcommand{\mfn}{\mathfrak{n}}
\newcommand{\mfp}{\mathfrak{p}}

\newcommand{\m}{\mathfrak{m}}

\newcommand{\p}{\mathfrak{p}}

\newcommand{\GL}{\mathrm{GL}}
\newcommand{\Ker}{\mathrm{Ker}}

\newcommand{\Tr}{\mathrm{Tr}}
\newcommand{\new}{\mathrm{new}}
\newcommand{\old}{\mathrm{old}}

\newcommand{\lra}{\longrightarrow}
\newcommand{\ra}{\rightarrow}

\def\1{1\!\!1}

\newcommand{\psmat}[4]{\bigl( \begin{smallmatrix} #1 & #2 \\ #3 & #4 \end{smallmatrix} \bigr)}

\newcommand{\Sla}{S^1_{k,m}({\mathrm{GL}}_2(A))}
\newcommand{\Slt}{S^1_{k,m}(\Gamma_0(t))}
\newcommand{\Sold}{S^{1,old}_{k,m}(\Gamma_0(t))}
\newcommand{\Snew}{S^{1,new}_{k,m}(\Gamma_0(t))}

\title[Drinfeld cusp forms of level $t$]{On Decomposition of Drinfeld cusp forms of level $t$}
\author[T. Dalal]{Tarun Dalal}
\email{tarun.dalal80@gmail.com}
\address{}

\date{\today}
\begin{document}
\begin{abstract}
In this article, we first prove that the Hecke operator $T_t$ has no eigenform in $\Sla$ with eigenvalue $-t^{k/2}$, when the characteristic of the base field is odd. Furthermore, if $\dim \Slt$ is even, we show that $T_t$ has no eigenform in $\Sla$ with eigenvalue $t^{k/2}$. As a consequence, we prove that the direct sum decomposition $\Slt=\Sold \oplus \Snew$ holds when $\dim \Slt$ is even. 
This proves the conjecture \cite[Conjecture 1.1(3)]{BV19a} of Bandini and Valentino for an infinite family of cusp forms. In particular, for any weight $k$, there exists at least one type $m$ (there are only two possible non-trivial values of $m$) for which the conjecture \cite[Conjecture 1.1(3)]{BV19a} is true.
\end{abstract}
\maketitle
\section{Introduction}
\label{Introduction}
The theory of oldforms and newforms is a well-developed area in the theory of classical modular forms. 
Many properties of modular forms heavily depend on whether they belong to the space of oldforms or newforms. 
For example, the space of newforms  has a basis consisting of normalized eigenforms for all the Hecke operators. Moreover, the field generated by the Fourier coefficients of a normalized newform is a number field. To the best of author's knowledge,  the analogues theory of oldforms and newforms is not yet completely known for Drinfeld modular forms. 

In a series of articles, Bandini and Valentino initiated the study towards Drinfeld oldforms and newform (cf. \cite{BV19}, \cite{BV19a}, \cite{BV20}, \cite{Val}). More precisely, in these series of articles, they defined the notion of $\mfp$-oldforms and $\p$-newforms and studied some of their properties. For this article, we restrict ourself to the case $\deg(\mfp)=1$ i.e., $\mfp=(t)$. Let $\Sla$ (resp., $\Slt$) denote the space of Drinfeld cusp forms for $\GL_2(A)$ (resp., for $\Gamma_0(t)$).

In~\cite{BV19}, the authors defined the notion of oldforms $S_{k,m}^{1,\old}(\Gamma_0(t))$ 
and newforms $S_{k,m}^{1,\new}(\Gamma_0(t))$ for $\Gamma_0(t)$ (which are also known as $t$-oldforms and $t$-newforms). 
In~\cite{BV19a}, the authors made the following conjectures (throughout the article we always assume that the characteristic of the base field is odd)
\begin{conj}(\cite[Conjecture 1.1]{BV19a}) 
\label{level T conjecture}
\begin{enumerate}
\item $\ker(T_t)=0,$ where $T_t$ is acting on $\Sla,$
\item $U_t$ is diagonalizable on $\Slt,$
\item $\Slt=S_{k,m}^{1,\old}(\Gamma_0(t))\oplus S_{k,m}^{1,\new}(\Gamma_0(t)).$
\end{enumerate}
\end{conj}
In~\cite{BV19a} and \cite{BV23}, using harmonic cocycles and matrix of the trace operator, the authors proved that Conjecture~\ref{level T conjecture} is true when $\dim \Sla\leq 1$.
In \cite{DK23}, using the explicit basis of $\Slt$ (constructed in \cite{DK2}), the authors proved that Conjecture \ref{level T conjecture}(3) is true when $\dim \Sla\leq 2$.

Recently in \cite{Vri26}, using the Ramanujan bound (developed in \cite{Vri25}) and the $A$-expansion of Drinfeld modular forms (cf. \cite{Pet13} and \cite{JP14} for the properties of $A$-expansion), Sjoerd de Vries proved Conjecture \ref{level T conjecture}(1) completely, and proved Conjecture \ref{level T conjecture}(3) under the assumption $\dim \Sla\leq p$ (moreover, the author proved these results for any prime ideal $\mfp$). 

An important observation is, for a fixed prime $p$, combining all these known results together, we can  confirm the validity of Conjecture \ref{level T conjecture}(3) only for finitely many values $k,m$.

The aim of this article is to prove that, for a fixed odd prime $p$, Conjecture \ref{level T conjecture}(3) is true for an infinite family of $k,m$.
In fact, we prove a more stronger result that, for any given $k$, there exists at least one $m$ such that Conjecture \ref{level T conjecture}(3) is true for $\Slt$. 
Roughly speaking, this proves that Conjecture \ref{level T conjecture}(3) is true for at least half of the cusp forms.

A common initial idea to prove Conjecture \ref{level T conjecture}(3) is to show that the Hecke operator $T_t$ has no eigenform in $\Sla$ with eigenvalues $\pm t^{\frac{k}{2}}$. 
In this article, we first completely rule out the case $-t^{k/2}$. More precisely, our first main result is the following:
\begin{thm}\label{Main Theorem 1 intro}
$T_t$ has no eigenform in $\Sla$ with eigenvalue $-t^{k/2}$.
\end{thm}
Next, under the assumption that $\dim \Slt$ is even, we rule out the case $t^{k/2}$. More precisely, we prove
\begin{thm}\label{Main Theorem 2 intro}
If $\dim \Slt$ is even, then $T_t$ has no eigenform in $\Sla$ with eigenvalue $t^{k/2}$.
\end{thm}
As a consequence of Theorem \ref{Main Theorem 1 intro} and Theorem \ref{Main Theorem 2 intro}, we prove
\begin{cor}\label{Direct sum corollary intro}
Let $k,m$ be such that $\dim \Slt$ is even. Then $$\Slt=\Sold \oplus \Snew.$$
\end{cor}
\begin{remark}
Since $q$ is odd, for a fixed $k\in \N$ the equation $k\equiv 2m \pmod {q-1}$ has only two distinct solutions in $\Z/(q-1)\Z$, one is $m$ and the other is $m+\frac{q-1}{2}$. Moreover, $\dim S^1_{k,m+\frac{q-1}{2}}(\Gamma_0(t))=\dim \Slt-1$ (cf. \cite[\S 3]{BV19a}). Hence for any $k\in \N$, there are exactly two possible values of $m$ such that $\dim \Slt>0$, and Conjecture \ref{level T conjecture}(3) is true for at least one of them.
\end{remark}

\subsection*{The remaining case for Conjecture \ref{level T conjecture}(3)}
Combining Theorem \ref{Main Theorem 1 intro}, Theorem \ref{Main Theorem 2 intro} and Corollary \ref{Direct sum corollary intro} we can say that, in order to prove Conjecture \ref{level T conjecture}(3), it suffices to show $T_t$ has no eigenform in $\Sla$ with eigenvalue $t^{k/2}$, when $\dim \Slt$ is odd.

We conclude this section by briefly discussing the equivalent statements of Conjecture \ref{level T conjecture} for higher level.
\begin{remark}
In recent years, there are some works to check the validity of the equivalent statements of Conjecture \ref{level T conjecture} for higher level. For example, in \cite{DK23}, by producing some infinitely many counter examples, the authors showed that when the level of the cusp forms is not prime, then with the current definition of $\p$-oldforms and $\p$-newforms for level $\p\m$ (cf. \cite{BV20}, \cite{Val} for the definition), the direct sum decomposition $S_{k,m}^{1}(\Gamma_0(\p\m)=S_{k,m}^{1,\p-\old}(\Gamma_0(\p\m))\oplus S_{k,m}^{1,\p-\new}(\Gamma_0(\p\m))$ may not hold. However, till the date, there are no such known counter examples when the level is prime. In fact, in \cite{Vri26}, Sjoerd de Vries proved that  $T_\p$ is injective on $\oplus_{k,m}\Sla$ for any prime ideal $\p$, furthermore he also proved that the direct sum decomposition $S_{k,m}^{1}(\Gamma_0(\p)=S_{k,m}^{1,old}(\Gamma_0(\p))\oplus S_{k,m}^{1,new}(\Gamma_0(\p))$ holds under the assumption $\dim \Sla<p$ (when the level is prime, we write $``\new", ``\old"$ instead of $``\p-\new", ``\p-\old"$). Hence it will be very interesting to check how far the methods of this paper can be adapted to check the validity of the direct sum decomposition $S_{k,m}^{1}(\Gamma_0(\p)=S_{k,m}^{1,old}(\Gamma_0(\p))\oplus S_{k,m}^{1,new}(\Gamma_0(\p))$.
\end{remark}

\subsection*{Notations}
Throughout the article, we use the following notations:
\begin{itemize}
\item Let $q$ be a power of an odd prime number $p$.
\item Let $k\in \N$ and $m\in \Z/(q-1)\Z$ such that $k\equiv 2m \pmod {q-1}$.
      Let $0 \leq m \leq q-2$ denote a lift of $m \in \Z/(q-1)\Z$. For simplicity of notations,
      we use $m$ to denote the integer as well as its class. 
\end{itemize} 
Let $\F_q$ denote the finite field of order $q$. Then we set $A :=\F_q[t]$ 
and
$K :=\F_q(t)$.
Let $K_\infty=\F_q((\frac{1}{t}))$ be the completion of $K$ 
with respect to the infinite place $\infty$ (corresponding to $\frac{1}{t}$-adic valuation) and denote by $\C_\infty$ the completion of an algebraic closure of $K_{\infty}$. Since we will work with some matrices which were computed in \cite{BV19a}, for the benefit of the readers we follow the notations of loc. cit. as closely as possible.

\section{Background Material}
\label{Basic Theory}

In this section, we briefly recall some basic facts about Drinfeld modular forms of level $t$ and introduce some operators which are needed to prove the main theorems
 (cf. \cite{Gos80}, \cite{Gos80a},  \cite{Gek88}, \cite{Tei91}, \cite{GR96}, \cite{BV19}, \cite{BV19a} for more details).

\subsection{Drinfeld modular forms}
The Drinfeld upper half-plane $\Omega:= \C_\infty-K_\infty$, which is analogue to the complex upper half-plane, has a rigid analytic structure, and 
the group $\GL_2(K_\infty)$ acts on $\Omega$ via fractional linear transformations.

\begin{dfn}
	Let $k\in \N$, $m \in \Z/(q-1)\Z$ and $f:\Omega \ra \C_\infty$ 
	be a rigid holomorphic function on $\Omega$. For any $\gamma=\psmat{a}{b}{c}{d}\in \GL_2(K_{\infty})$,
	we define the slash operator $|_{k,m} \gamma$ on $f$ by
	\begin{equation}\label{slash operator}
	f|_{k,m} \gamma := (\det \gamma)^{m}(cz+d)^{-k}f(\gamma z).
	\end{equation}
\end{dfn}
For an ideal $\mfn \subseteq A$, let $\Gamma_0(\mfn)$ denote the congruence subgroup 
$\{\psmat{a}{b}{c}{d}\in \GL_2(A): c\in \mfn \}$.
A Drinfeld modular form of weight $k$, type $m$ for $\Gamma_0(\mfn)$ is defined as follows:
\begin{dfn}
	\label{Definition of DMF}
	A rigid holomorphic function $f:\Omega \ra \C_\infty$ is said to be a Drinfeld modular form of weight $k$, type $m$ 
	for $\Gamma_0(\mfn)$ if 
	\begin{enumerate}
		\item $f|_{k,m}\gamma= f$ , $\forall \gamma\in \Gamma_0(\mfn)$,
		\item $f$ is holomorphic at the cusps of $\Gamma_0(\mfn)$.
	\end{enumerate}
	Furthermore, if $f$ vanishes at the cusps of $\Gamma_0(\mfn)$, then we say $f$ is a Drinfeld cusp form of
	weight $k$, type $m$ for $\Gamma_0(\mfn)$, and the space of such forms is denoted by $S^1_{k,m}(\Gamma_0(\mfn))$.
\end{dfn}
Recall that if $k\not \equiv 2m \pmod {q-1}$, then $S^1_{k,m}(\Gamma_0(\mfn))=\{0\}$. So, without loss of generality, we can assume that $k\equiv 2m \pmod {q-1}$. We now restrict ourself to the case $\mfn=(t)$. 
\subsection{Harmonic cocycles for $\Gamma_0(t)$}
For fixed $k,m$ with $k\equiv 2m\pmod {q-1}$ we define 
$$j\in \{0,1,\ldots,q-2\} \ \text{such that} \ m\equiv j+1\pmod {q-1}, \ n:=\frac{k-2(j+1)}{q-1}+1.$$ 
Then $k=2(j+1)+(n-1)(q-1)$ and $n=\dim \Slt$.
Recall that the space $\Slt$ can be identified with the block $C_j$ of harmonic cocycles with basis $\mathbf{c}_{j+(i-1)(q-1)}, 1\leq i\leq n$ as mentioned in \cite[\S 3]{BV19a}. We will not discuss this relation here since we only work with some matrices that were already computed in loc. cit. with respect to this basis of $C_j$ (we strongly refer the reader to \cite{Tei91} for the relation between Drinfeld cusp forms and Harmonic cocycles).
\subsection{Hecke operators:}
Next, we recall the definitions of $T_t$ and $U_t$-operators. 

\begin{dfn}
\label{Definitions of T_p and U_p operators}
For $f \in \Sla$, we define
\begin{equation}
T_t(f) := t^{k-m} \sum_{\zeta\in \F_q}f|_{k,m}\psmat{1}{\zeta}{0}{t} + t^{k-m}f|_{k,m}\psmat{t}{0}{0}{1}.
\end{equation}
For any $f \in \Slt$, we define
\begin{equation}
U_t(f) := t^{k-m} \sum_{\zeta\in \F_q}f|_{k,m}\psmat{1}{\zeta}{0}{t}.
\end{equation}
With respect to the basis $\{\mathbf{c}_{j+(i-1)(q-1)}\}$, the matrix of $U_t$ is $MD:=(m_{a,b})_{n\times n}\biggl(\begin{smallmatrix}
t^{s_1}&\cdots &0\\
\vdots &\ddots &\vdots\\
0 & \cdots &t^{s_n}
\end{smallmatrix}\biggr)$, where
\begin{equation}\label{eq: value of m_ab}
m_{a,b}:=
\begin{cases}
-\bigg[\binom{j+(n-a)(q-1)}{j+(n-b)(q-1)}+(-1)^{j+1}\binom{j+(n-a)(q-1)}{j+(b-1)(q-1)} \bigg], \ \text{if} \ a\neq b\\
(-1)^j\binom{j+(n-a)(q-1)}{j+(a-1)(q-1)}, \ \text{if} \ a=b,
\end{cases}
\end{equation}
and $s_i:=j+1+(i-1)(q-1)$ for $1\leq i \leq n$ (cf. \cite[(11),(12)]{BV19a}) (note that $m_{a,b}\in \F_p$).
\end{dfn}
\subsection{Fricke involution and Trace operator} The Fricke involution $W_t$ is defined by the action of the matrix $\psmat{0}{-1}{t}{0}$ on $S_{k,m}^1(\Gamma_0(t))$ via the slash operator. 
\begin{dfn}
The trace operator $\Tr$ is defined as
\begin{equation*}
\Tr : \Slt \longrightarrow \Sla  \ \mathrm{with} \ \Tr (f) = \sum_{\gamma\in \Gamma_0(t)\char`\\ \GL_2(A)} f|_{k,m}\gamma.
\end{equation*}
For any $f\in S^1_{k,m}(\Gamma_0(t))$, we have (cf. \cite[(2-7)]{BV19a} or \cite[Proposition 3.8]{DK23} for more general result)
\begin{equation}
\Tr(f)=f+ t^{-m}U_t(f|W_t).
\end{equation}
\end{dfn}
From \cite[(14)]{BV19a}, we know that the matrix associated to the action of $W_t$ on $C_j$ is given by
\begin{equation}
t^{m-k}F:=t^{m-k}\biggl(\begin{smallmatrix}
0&\cdots &(-t)^{s_n}\\
\vdots &\iddots &\vdots\\
(-t)^{s_1} & \cdots &0
\end{smallmatrix}\biggr).
\end{equation}
The matrix associated to $\Tr$ on $C_j$ is given by (cf. \cite[(15)]{BV19a})
\begin{equation}\label{matrix of trace operator}
T:=I+M\mcA,\ \text{where} \ 
\mcA:=\biggl(\begin{smallmatrix}
0&\cdots &(-1)^{j+1}\\
\vdots &\iddots &\vdots\\
(-1)^{j+1} & \cdots &0
\end{smallmatrix}\biggr).
\end{equation}
Since $q$ is odd, we have $(-1)^{j+1+(i-1)(q-1)}=(-1)^{j+1}$ and   
\begin{equation}\label{A^2=A}
F=\mcA D \  \text{and} \ \mcA^2=I.
\end{equation}

\subsection{Oldforms and Newforms}\label{definition of newforms oldforms}
We now recall the definition of oldforms and newform from \cite{BV19}.
Consider the map
\begin{align*}
\delta: (\Sla)^2 & \lra \Slt\\
(f,g) & \lra \delta_1f + \delta_t g,
\end{align*} 
where 
$\delta_1$, $\delta_t : \Sla \ra \Slt$ defined by 
$\delta_1(f)=f$ and $\delta_t(f)= f|_{k,m}\psmat{t}{0}{0}{1}.$ By \cite[Proposition 3.1]{BV19}, the mappings $\delta_1,\delta_t$ are injective.

\begin{dfn}
The space of oldforms of level $t$ is defined as $\Sold:=Im(\delta)$.
\end{dfn}

\begin{dfn}
The space of newforms of level $t$ is defined as  
$$\Snew:= \Ker(\Tr)\cap \Ker(\Tr^\prime), \quad \mathrm{where} \ \Tr^\prime(f) : = \Tr(f|W_t).$$
\end{dfn}
\section{Moving from direct sum to  Eigenvalues of the Hecke operator $T_t$}
We recall the following key result (cf. \cite[Remark 5.3]{BV19a} or \cite[Proposition 4.5]{DK23} for a more general statement) which transforms Conjecture \ref{level T conjecture}(3) into proving the non existence of certain eigenforms of $T_t$ operator.
\begin{prop}
\label{Direct sum t case}
The direct sum decomposition
\begin{equation}\label{direct sum decomposition level t in corollary}
\Slt=\Sold \oplus \Snew
\end{equation}
 holds if and only if the $T_t$-operator has no eigenform in $\Sla$ with eigenvalues $\pm t^{\frac{k}{2}}$. 
\end{prop}
Since the matrix corresponding to $T_t$ is difficult to handle, our first aim is to find a relation between $T_t$ operator and some matrix which is relatively easy to handle. We start by giving an alternative description of the matrix $T$ in terms of $M$ and $\mcA$.
For $1\leq a,b\leq n$, define
$$\eta(a,b):=\binom{j+(n-a)(q-1)}{j+(n-b)(q-1)},\ \text{and} \ \mathbf{1}(a,b):=
\begin{cases}
1, \ \text{if} \ a=b\\
0, \ \text{otherwise}.
\end{cases}$$
With these new notations, from \eqref{eq: value of m_ab} it is easy to check that 
\begin{equation}
m_{a,b}=\mathbf{1}(a,b)-\eta(a,b)-(-1)^{j+1}\eta(a,n+1-b), \ \text{for} \ 1\leq a,b \leq n.
\end{equation}
The following result can be found in \cite[(6)]{BV23}, but we choose to give an explicit proof without distinguishing the odd and even case.
\begin{prop}\label{another matrix of trace operator}
$T=I+M\mcA=M+\mcA$.
\end{prop}
\begin{proof}
We prove the equality $I+M\mcA=M+\mcA$ entry wise.
For $1\leq a,b \leq n$ we have
\begin{align*}
(M\mcA)_{a,b}&=\sum_{c=1}^nm_{a,c}(\mcA)_{c,b}=(-1)^{j+1}m_{a,n+1-b}\\
&= (-1)^{j+1}\{\mathbf{1}(a,n+1-b)-\eta(a,n+1-b)-(-1)^{j+1}\eta(a,b)\}\\
&=(-1)^{j+1}\mathbf{1}(a,n+1-b)-(-1)^{j+1}\eta(a,n+1-b)-\eta(a,b).
\end{align*}
\begin{align*}
\mathrm{Hence} \ (I+M\mcA)_{a,b}&=\mathbf{1}(a,b)+(M\mcA)_{a,b}\\
&=\mathbf{1}(a,b)+(-1)^{j+1}\mathbf{1}(a,n+1-b)-(-1)^{j+1}\eta(a,n+1-b)-\eta(a,b)\\
&=\mathbf{1}(a,b)-\eta(a,b)-(-1)^{j+1}\eta(a,n+1-b)+(-1)^{j+1}\mathbf{1}(a,n+1-b)\\
&=m_{a,b}+(\mcA)_{a,b}\\
&=(M+\mcA)_{a,b}.
\end{align*}
The result follows.
\end{proof}
We recall the following results from \cite{BV19a}.
\begin{prop}\label{some properties of T}
\begin{enumerate}
\item $T^2=T$.
\item $T\mcA=T$.
\item  $Im(\delta_1)=Im(T)=\ker(T-I)$. \item $r:=rank(T)=\dim\Sla$.
\end{enumerate}
\end{prop}
\begin{proof}
$(1)$ and $(3)$ follows from \cite[\S 4.3, (4-6)]{BV19a}. 
On the other hand, from \eqref{A^2=A} and Proposition \ref{another matrix of trace operator} we have $T\mcA=(I+M\mcA)\mcA=\mcA+M=T$, this proves $(2)$.  $(4)$ follows from $(3)$ and the fact that $\delta_1$ is injective (cf. \S\ref{definition of newforms oldforms}). 
\end{proof}
The following result plays a crucial role in proving the main theorems.
\begin{lem}\label{T_t to trace determinant}
For any $\epsilon \in \{\pm 1\}$, $T_t$ has no eigenform in $\Sla$ with eigenvalue $\epsilon t^{k/2}$ if and only if $\det(t^{-k/2}TD-\epsilon I)\neq 0$ in $\F_p[t,t^{-1}]$. 
\end{lem}
\begin{proof}
Since $t^{k/2}TD\in M_{n\times n}(\F_p[t,t^{-1}])$, we have
$\det(t^{-k/2}TD-\epsilon I)\neq 0$ in $\F_p[t,t^{-1}]$ if and only if $\det(t^{-k/2}TD-\epsilon I)\neq 0$ in $\C_\infty$.
Recall that we have the isomorphism $\delta_1:\Sla\cong Im(T)=Im(\delta_1)$.
From \cite[Remark 5.3]{BV19a}, \eqref{A^2=A} and Proposition \ref{some properties of T}, we get 
\begin{equation}\label{eq: T_t and TD}
\delta_1(T_t f)=TF\delta_1(f)=T\mcA D \delta_1(f)=TD\delta_1(f), \ \text{for} \ f \in \Sla.
\end{equation}
Let $0\neq v\in \Slt$ be an eigenform of $t^{-k/2}TD$ with eigenvalue $\epsilon$. Then 
\begin{equation}\label{eq:4.1v1}
v=\epsilon^{-1}t^{-k/2}TD(v)\in Im(T)=Im(\delta_1).
\end{equation}
Now first consider $\det(t^{-k/2}TD-\epsilon I)=0$. Then $t^{-k/2}TD$ has a nonzero eigenform $v$ corresponding to the eigenvalue $\epsilon$, and by \eqref{eq:4.1v1}, $\exists f\in \Sla$ such that $v=\delta_1(f)$ (since $\delta_1$ is injective and $v\neq 0$, it follows that $f\neq 0$). Now (loc. cit.)
\begin{equation}
\delta_1(T_tf)=TD\delta_1(f)=\epsilon t^{k/2}\delta_1(f)=\delta_1(\epsilon t^{k/2}f).
\end{equation}
Since $\delta_1$ is injective, we have $T_tf=\epsilon t^{k/2}f$, i.e., $T_t$ has an eigenform in $\Sla$ with eigenvalue $\epsilon t^{k/2}$.

For the converse part, assume that $T_t$ has an eigenform $f\in \Sla$ with eigenvalue $\epsilon t^{k/2}$, i.e., $\delta_1(T_tf)=\epsilon t^{k/2}\delta_1(f)$. From \eqref{eq: T_t and TD}, we get
\begin{equation}
TD\delta_1(f)=\epsilon t^{k/2}\delta_1(f), \ \text{equivalently} \ t^{-k/2}TD\delta_1(f)=\epsilon\delta_1(f).
\end{equation}
Therefore $t^{-k/2}TD$ has an eigenform corresponding to the eigenvalue $\epsilon$. Consequently we have  
$\det(t^{-k/2}TD-\epsilon I)\neq 0$. Hence we proved that $T_t$ has an eigenform in $\Sla$ with eigenvalue $\epsilon t^{k/2}$ if and only if $\det(t^{-k/2}TD-\epsilon I)=0$ in $\F_p[t,t^{-1}]$. The result follows.
\end{proof}
We now note a key observation which allows us to check the vanishing of the determinant $\det(t^{-k/2}TD\pm I)$ relatively easy compared to working with explicit matrix computations.
\begin{prop}\label{det in polynomial ring and in C}
Let $C(X)\in M_{n\times n}(\F_p[X,X^{-1}])$. If
$\det C(\zeta)\neq 0 \ \text{in }\overline{\F_p}\ \text{for some }\zeta\in\overline{\F_p}^{\times},$
then $\det C(t)\ne0$ in $\C_\infty$.
\end{prop}
\begin{proof}
Recall that for any commutative ring with unity, any ring homomorphism $\F_p[X,X^{-1}]\to R$ naturally extends to an algebra homomorphism $M_{n\times n}(\F_p[X,X^{-1}])\to M_{n\times n}(R)$ which commutes with the determinant. 

Hence under the natural mapping $\F_p[X,X^{-1}]\to \F_p[\zeta,\zeta^{-1}]$ (for some $\zeta \in\overline{\F_p}^{\times}$), if $C(\zeta)\neq 0$, then $\det(C(X))\neq 0$. Consequently, under the natural mapping $\F_p[X,X^{-1}]\to \C_\infty$ ( defined by $X\to t$), we have $\det(C(t))\neq 0$.
\end{proof}
Since $T$ is idempotent matrix of rank $r$, $T$ is diagonalizable with eignvalues $1$ (with multiplicity $r$), $0$ (with multiplicity $n-r$) and we have 
\begin{equation}\label{eq:4.3}
\det(T-\epsilon I)=(1-\epsilon)^r(-\epsilon)^{n-r}, \ \text{and} \
\det(-T-\epsilon I)=(-1-\epsilon)^r(-\epsilon)^{n-r}, \ \text{for} \ \epsilon\in \{\pm 1\}.
\end{equation}
Let $D^\prime(X)=X^{-k/2}\biggl(\begin{smallmatrix}
X^{s_1}&\cdots &0\\
\vdots &\ddots &\vdots\\
0 & \cdots &X^{s_n}
\end{smallmatrix}\biggr)\in M_{n\times n}(\F_p[X,X^{-1}])$. Then $D^\prime(t)=t^{-k/2}D$ and $D^\prime(1)=I$. We are now ready to prove the main results of the article.
\section{Proof of Theorem \ref{Main Theorem 1 intro}}
\begin{thm}\label{Main Theorem 1, -Eigenvalue}
$T_t$ has no eigenform in $\Sla$, with eigenvalue $-t^{k/2}$.
\end{thm}
\begin{proof}
By Lemma \ref{T_t to trace determinant}, we need to show that $\det(t^{-k/2}TD+ I)=\det(TD^\prime(t)+ I)\neq 0$.

Now by \eqref{eq:4.3} we have $\det(TD^\prime(1)+ I)=\det(T+ I)=2^r\neq 0.$ Hence from Proposition \ref{det in polynomial ring and in C}, we conclude that $\det(t^{-k/2}TD+ I)=\det(TD^\prime(t)+ I)\neq 0$. The result follows.
\end{proof}

\section{Proof of Theorem \ref{Main Theorem 2 intro}}
\begin{thm}\label{Main Theorem 2,+Eigenvalue}
If $n$ is even, then $T_t$ has no eigenform in $\Sla$, with eigenvalue $t^{k/2}$.
\end{thm}
\begin{proof}
By Lemma \ref{T_t to trace determinant}, we need to show that $\det(t^{-k/2}TD- I)=\det(TD^\prime(t)- I)\neq 0$.

Let $\zeta\in \F_q^\times$ such that $\mathrm{order}(\zeta)=q-1$. Since $n$ is even, $2(i-1)+n-1$ is odd for $1\leq i \leq n$ and
\begin{equation*}
\zeta^{s_i-\frac{k}{2}}=\zeta^{(2(i-1)+n-1)\frac{q-1}{2}}=(-1)^{(2(i-1)+n-1)}=-1.
\end{equation*}
Hence $D^\prime(\zeta)=-I_n.$
Now by \eqref{eq:4.3} we have $\det(TD^\prime(\zeta)- I)=\det(-T- I)=(-2)^r(-1)^{n-r}\neq 0.$ Hence from Proposition \ref{det in polynomial ring and in C}, we conclude that $\det(t^{-k/2}TD- I)=\det(TD^\prime(t)- I)\neq 0$. The result follows.
\end{proof}
\section{Proof of Corollary \ref{Direct sum corollary intro}}
\begin{cor}
Let $k,m$ be such that $n=\dim \Slt$ is even. Then $$\Slt=\Sold \oplus \Snew.$$
\end{cor}
\begin{proof}
By Proposition \ref{Direct sum t case}, it suffices to show that for such values of $k,m$, the operator $T_t$ has no eigenform in $\Sla$ with eigenvalue $\pm t^{k/2}$. 

From Theorem \ref{Main Theorem 1, -Eigenvalue}, $T_t$ has no eigenform in $\Sla$ with eigenvalue $- t^{k/2}$. 
Furthermore, since $n$ is even, from Theorem \ref{Main Theorem 2,+Eigenvalue} we conclude that $T_t$ has no eigenform in $\Sla$ with eigenvalue $t^{k/2}$. Thus $T_t$ has no eigenform in $\Sla$ with eigenvalue $\pm t^{k/2}$ when $n$ is even. This completes the proof. 
\end{proof}
\bibliographystyle{plain, abbrv}

\end{document}